\documentclass[11pt,a4paper]{article}%
\usepackage{amsmath}
\usepackage{amsfonts}
\usepackage{amssymb}
\usepackage{graphicx}%
\providecommand{\U}[1]{\protect\rule{.1in}{.1in}}
\providecommand{\U}[1]{\protect\rule{.1in}{.1in}}
\newtheorem{theorem}{Theorem}

\newtheorem{corollary}[theorem]{Corollary}

\newtheorem{definition}[theorem]{Definition}
\newtheorem{example}{Example}

\newtheorem{proposition}[theorem]{Proposition}

\newenvironment{proof}[1][Proof]{\noindent\textbf{#1.} }{\ \rule{0.5em}{0.5em}}
\include {mak}
\begin{document}

\begin{center}
{\Large The partial }$r${\Large -Bell polynomials}

\ \ \ 

{\large Miloud Mihoubi} and Mourad Rahmani

USTHB, Faculty of Mathematics, Po. Box 32 El Alia 16111 Algiers, Algeria.

mmihoubi@usthb.dz \ or \ miloudmihoubi@gmail.com

{\large mrahmani@usthb.dz}
\end{center}

$\ \ \ \newline$\noindent\textbf{Abstract.} In this paper, we show that the
$r$-Stirling numbers of both kinds, the $r$-Whitney numbers of both kinds, the
$r$-Lah numbers\ and the $r$-Whitney-Lah numbers form particular cases of
family of polynomials forming a generalization of the partial Bell
polynomials. We deduce the generating functions of several restrictions of
these numbers. In addition, a new combinatorial interpretations is presented
for the $r$-Whitney numbers and the $r$-Whitney-Lah numbers.



\section{Introduction}

The exponential partial Bell polynomials $B_{n,k}\left(  x_{1},x_{2}%
,\ldots\right)  :=B_{n,k}\left(  x_{j}\right)  $ in an infinite number of
variables $x_{j},$ $\left(  j\geq1\right)  ,$ introduced by Bell \cite{bell},
as a mathematical tool for representing the $n$-th derivative of composite
function. These polynomials are often used in combinatorics, statistics and
also mathematical applications. They are defined by their generating function%
\[
\underset{n\geq k}{\sum}B_{n,k}\left(  x_{j}\right)  \dfrac{t^{n}}{n!}%
=\dfrac{1}{k!}\left(  \underset{m\geq1}{\sum}x_{m}\dfrac{t^{m}}{m!}\right)
^{k},
\]
and are given explicitly by the formula%
\begin{equation}
B_{n,k}\left(  a_{1},a_{2},\ldots\right)  =\underset{\pi\left(  n,k\right)
}{\sum}\frac{n!}{k_{1}!\cdots k_{n}!}\left(  \frac{a_{1}}{1!}\right)  ^{k_{1}%
}\left(  \frac{a_{2}}{2!}\right)  ^{k_{2}}\cdots\left(  \frac{a_{n}}%
{n!}\right)  ^{k_{n}}, \label{2}%
\end{equation}
where%
\[
\pi\left(  n,k\right)  =\left\{  \mathbf{k=}\left(  k_{1},\ldots,k_{n}\right)
\in\mathbb{N}^{n}:k_{1}+k_{2}+\cdots+k_{n}=k,\ \ k_{1}+2k_{2}+\cdots
+nk_{n}=n\right\}  .
\]
It is well-known that for appropriate choices of the variables $x_{j},$ the
exponential partial Bell polynomials reduce to some special combinatorial
sequences. We mention the following special cases:%
\begin{align*}%
\genfrac{[}{]}{0pt}{}{n}{k}%
&  =B_{n,k}\left(  0!,1!,2!,\cdots\right)  ,\ \text{unsigned Stirling numbers
of the first kind,}\\%
\genfrac{\{}{\}}{0pt}{}{n}{k}%
&  =B_{n,k}\left(  1,1,1,\ldots\right)  ,\text{ Stirling numbers of the second
kind,}\\%
\genfrac{\lfloor}{\rfloor}{0pt}{}{n}{k}%
&  =B_{n,k}\left(  1!,2!,3!,\cdots\right)  ,\text{ Lah numbers,}\\
\binom{n}{k}k^{n-k}  &  =B_{n,k}\left(  1,2,3,\cdots\right)  ,\text{
idempotent numbers.}%
\end{align*}
For more details on these numbers, one can see \cite{bell, com, mih1,
mih2,wang}.

\noindent In 1984, Broder \cite{bro} generalized the Stirling numbers of both
kinds to the so-called $r$-Stirling numbers. In this paper, after recalling
the partition polynomials, we give a unified method for obtaining a class of
special combinatorial sequences, called the exponential partial $r $-Bell
polynomials for which the $r$-Stirling numbers and other known numbers appear
as special cases. In addition, these polynomials generalize the exponential
partial Bell polynomials and posses some combinatorial interpretations in
terms of set partitions.

\section{The partial $r$-Bell polynomials}

First of all, to introduce the partial $r$-Bell polynomials, we may give some
combinatorial interpretations of the partial Bell polynomials. Below, for
$B_{n,k}\left(  a_{1},a_{2},a_{3},\ldots\right)  ,$ we use $B_{n,k}\left(
a_{l}\right)  $ and sometimes we use $B_{n,k}\left(  a_{1},a_{2},a_{3}%
,\ldots\right)  $ and for $B_{n,k}^{\left(  r\right)  }\left(  a_{1}%
,a_{2},\ldots;b_{1},b_{2},\ldots\right)  ,$ we use $B_{n,k}^{\left(  r\right)
}\left(  a_{l};b_{l}\right)  $ and sometimes we use $B_{n,k}^{\left(
r\right)  }\left(  a_{1},a_{2},\ldots;b_{1},b_{2},\ldots\right)  .$

\begin{theorem}
\label{T0}Let $\left(  a_{n};n\geq1\right)  $ be a sequence of nonnegative
integers. Then, we have

\begin{itemize}
\item the number $B_{n,k}\left(  a_{l}\right)  $ counts the number of
partitions of a $n$-set into $k$ blocks such that the blocks of the same
cardinality $i$ can be colored with $a_{i}$ colors,

\item the number $B_{n,k}\left(  \left(  l-1\right)  !a_{l}\right)  $ counts
the number of permutations of a $n$-set into $k$ cycles such that any cycle of
length $i$ can be colored with $a_{i}$ colors, and,

\item the number $B_{n,k}\left(  l!a_{l}\right)  $ counts the number of
partitions of a $n$-set into $k$ ordered blocks such that the blocks of
cardinality $i$ can be colored with $a_{i}$ colors.
\end{itemize}
\end{theorem}

\begin{proof}
For a partition of a finite $n$-set that is decomposed into $k$ blocks, let
$k_{i}$ be the number of blocks of the same cardinality $i,$ $i=1,\ldots,n.$
Then, the number to choice such partition is%
\[
\frac{n!}{k_{1}!\left(  1!\right)  ^{k_{1}}k_{2}!\left(  2!\right)  ^{k_{2}%
}\cdots k_{n}!\left(  n!\right)  ^{k_{n}}},\text{ \ \ }\mathbf{k=}\left(
k_{1},\ldots,k_{n}\right)  \in\pi\left(  n,k\right)  ,
\]
and, the number to choice such partition for which the blocks of the same
cardinality $i$ can be colored with $a_{i}$ colors is%
\[
\frac{n!}{k_{1}!\left(  1!\right)  ^{k_{1}}k_{2}!\left(  2!\right)  ^{k_{2}%
}\cdots k_{n}!\left(  n!\right)  ^{k_{n}}}\left(  a_{1}\right)  ^{k_{1}%
}\left(  a_{2}\right)  ^{k_{2}}\cdots\left(  a_{n}\right)  ^{k_{n}},\text{
\ \ }\mathbf{k=}\left(  k_{1},\ldots,k_{n}\right)  \in\pi\left(  n,k\right)
,
\]
Then, the number of partitions of a $n$-set into $k$ blocks of cardinalities
$k_{1},k_{2},\ldots,k_{n}$ such that the blocks of the same length $i$ can be
colored with $a_{i}$ colors is%
\[
\underset{\mathbf{k}\in\pi\left(  n,k\right)  }{\sum}\frac{n!}{k_{1}!\left(
1\right)  ^{k_{1}}\cdots k_{n}!\left(  n\right)  ^{k_{n}}}\left(
a_{1}\right)  ^{k_{1}}\left(  a_{2}\right)  ^{k_{2}}\cdots\left(
a_{n}\right)  ^{k_{n}}=B_{n,k}\left(  a_{l}\right)  .
\]
For the combinatorial interpretations of $B_{n,k}\left(  \left(  l-1\right)
!a_{l}\right)  $ and $B_{n,k}\left(  l!a_{l}\right)  ,$ we can proceed
similarly as above.
\end{proof}

\begin{definition}
\label{D1}Let $\left(  a_{n};n\geq1\right)  $ and $\left(  b_{n}%
;n\geq1\right)  $ be two sequences of nonnegative integers. The number
$B_{n+r,k+r}^{\left(  r\right)  }\left(  a_{l};b_{l}\right)  $ counts the
number of partitions of a $\left(  n+r\right)  $-set into $\left(  k+r\right)
$ blocks such that:

\begin{itemize}
\item the $r$ first elements are in different blocks,

\item any block of the length $i$ with no elements of the $r$ first elements,
can be colored with $a_{i}$ colors,

\item any block of the length $i$ with one element of the $r$ first elements,
can be colored with $b_{i}$ colors.\newline We assume that any block with $0$
color does not appear in partitions.
\end{itemize}
\end{definition}

\noindent On using this definition, the following theorem gives an interesting
relation which help us to find a family of polynomials generalize the above numbers.

\noindent On using combinatorial arguments, the partial $r$-Bell polynomials
admit the following expression.

\begin{theorem}
\label{TT1}For $n\geq k\geq r\geq1,$ the partial $r$-Bell polynomials can be
written as%
\[
B_{n,k}^{\left(  r\right)  }\left(  a_{1},a_{2},\ldots;b_{1},b_{2}%
,\ldots\right)  =\frac{\left(  n-r\right)  !}{\left(  k-r\right)  !}%
\underset{n_{1}+\cdots+n_{k}=n+r-k}{\sum}\frac{b_{n_{1}+1}\cdots b_{n_{r}+1}%
}{n_{1}!\cdots n_{r}!}\frac{a_{n_{r+1}+1}\cdots a_{n_{k}+1}}{\left(
n_{r+1}+1\right)  !\cdots\left(  n_{k}+1\right)  !}.
\]

\end{theorem}

\begin{proof}
Consider the $\left(  n+r\right)  $-set as union of two sets $\mathbf{R}$
which contains the $r$ first elements and $\mathbf{N}$ which contains the $n$
last elements. To partition a $\left(  n+r\right)  $-set into $k+r$ blocks
$B_{1},\ldots,B_{k+r}$ given as in Definition \ref{D1}, let the elements of
$\mathbf{R}$ be in different $r$ blocks $B_{1},\ldots,B_{r}.$ \newline There
is $\frac{1}{k!}\binom{n}{n_{1},\ldots,n_{k+r}}b_{n_{1}+1}\cdots b_{n_{r}%
+1}a_{n_{r+1}}\cdots a_{n_{r+k}}$ ways to choose $n_{1},\ldots,n_{k+r}$ in
$\mathbf{N}$ on using colors, such that\newline- $n_{1}\geq0,\ldots,n_{r}%
\geq0:$ $n_{1},\ldots,n_{r}$ to be, respectively, in $B_{1},\ldots,B_{r}$ with
$b_{n_{1}+1}\cdots b_{n_{r}+1}$ ways to color these blocks,\newline-
$n_{r+1}\geq1,\ldots,n_{k+r}\geq1:$ $n_{r+1},\ldots,n_{k+r}$ to be,
respectively, in $B_{r+1},\ldots,B_{k+r}$ with $\frac{1}{k!}a_{n_{r+1}}\cdots
a_{n_{r+k}}$ ways to color these blocks.

\noindent Then, the total number of colored partitions is%
\[
B_{n+r,k+r}^{\left(  r\right)  }\left(  a_{1},a_{2},\ldots;b_{1},b_{2}%
,\ldots\right)  =\frac{1}{k!}\underset{\left(  n_{1},\ldots,n_{k+r}\right)
\in M_{n+r,k+r}}{\sum}\binom{n}{n_{1},\ldots,n_{k+r}}b_{n_{1}+1}\cdots
b_{n_{r}+1}a_{n_{r+1}}\cdots a_{n_{r+k}},
\]
where $M_{n,k}=\left\{  \left(  n_{1},\ldots,n_{k}\right)  :n_{1}+\cdots
+n_{k}=n,\ \left(  n_{1},\ldots,n_{r},n_{r+1}-1,\ldots,n_{k}-1\right)
\in\mathbb{N}^{k}\right\}  .$
\end{proof}

\noindent On using Theorem \ref{TT1}, we may state that:

\begin{corollary}
\label{T1}We have%
\begin{equation}
\underset{n\geq k}{\sum}B_{n+r,k+r}^{\left(  r\right)  }\left(  a_{l}%
;b_{l}\right)  \frac{t^{n}}{n!}=\frac{1}{k!}\left(  \underset{j\geq1}{\sum
}a_{j}\frac{t^{j}}{j!}\right)  ^{k}\left(  \underset{j\geq0}{\sum}b_{j+1}%
\frac{t^{j}}{j!}\right)  ^{r}. \label{3}%
\end{equation}

\end{corollary}

\begin{proof}
From Theorem \ref{TT1} we get%
\begin{align*}
&  \underset{n\geq k}{\sum}B_{n+r,k+r}^{\left(  r\right)  }\left(  a_{l}%
;b_{l}\right)  \frac{t^{n}}{n!}\\
&  =\underset{n\geq k}{\sum}\left(  \frac{1}{k!}\underset{n_{1}+\cdots
+n_{r+k}=n+r-k}{\sum}\frac{b_{n_{1}+1}\cdots b_{n_{r}+1}}{n_{1}!\cdots n_{r}%
!}\frac{a_{n_{r+1}+1}\cdots a_{n_{r+k}+1}}{\left(  n_{r+1}+1\right)
!\cdots\left(  n_{r+k}+1\right)  !}\right)  t^{n}\\
&  =\frac{1}{k!}\underset{n_{1}\geq0,\ldots,n_{r}\geq0,\ n_{r+1}\geq
1,\ldots,n_{r+k}\geq1}{\sum}\frac{b_{n_{1}+1}\cdots b_{n_{r}+1}}{n_{1}!\cdots
n_{r}!}\frac{a_{n_{r+1}}\cdots a_{n_{r+k}}}{n_{r+1}!\cdots n_{r+k}%
!}t^{t^{n_{1}+\cdots+n_{r+k}}}\\
&  =\frac{1}{k!}\left(  \underset{j\geq1}{\sum}a_{j}\frac{t^{j}}{j!}\right)
^{k}\left(  \underset{j\geq0}{\sum}b_{j+1}\frac{t^{j}}{j!}\right)  ^{r}.
\end{align*}

\end{proof}

\noindent To give an explicit expression of the number $B_{n+r,k+r}^{\left(
r\right)  }\left(  a_{l};b_{l}\right)  $ generalizing the formula (\ref{2}),
we use the Touchard polynomials defined in \cite{chr} as follows. Let $\left(
x_{i};i\geq1\right)  $ and $\left(  y_{i};i\geq1\right)  $ be two sequences of
indeterminates, the Touchard polynomials%
\[
T_{n,k}\left(  x_{j},y_{j}\right)  \equiv T_{n,k}\left(  x_{1},\ldots
,x_{n};y_{1},\ldots,y_{n}\right)  ,\text{ }n=k,k+1,\cdots,
\]
are defined by $T_{0,0}=1$ and the sum%
\[
T_{n,k}\left(  x_{1},x_{2},\ldots;y_{1},y_{2},\ldots\right)  =\underset
{\Lambda\left(  n,k\right)  }{\sum}\left[  \frac{n!}{k_{1}!k_{2}!\cdots
}\left(  \frac{x_{1}}{1!}\right)  ^{k_{1}}\left(  \frac{x_{2}}{2!}\right)
^{k_{2}}\cdots\right]  \left[  \frac{1}{r_{1}!r_{2}!\cdots}\left(  \frac
{y_{1}}{1!}\right)  ^{r_{1}}\left(  \frac{y_{2}}{2!}\right)  ^{r_{2}}%
\cdots\right]  ,
\]
where%
\[
\Lambda\left(  n,k\right)  =\left\{  \mathbf{k=}\left(  k_{1},k_{2}%
,\ldots\right)  :k_{i}\in\mathbb{N},\text{ }i\geq1,\ \underset{i\geq1}{\sum
}k_{i}=k,\ \underset{i\geq1}{\sum}i\left(  k_{i}+r_{i}\right)  =n\right\}  ,
\]
and admits a vertical generating function given by%
\begin{equation}
\underset{n=k}{\overset{\infty}{\sum}}T_{n,k}\left(  x_{1},x_{2},\ldots
;y_{1},y_{2},\ldots\right)  \frac{t^{n}}{n!}=\frac{1}{k!}\left(
\underset{i\geq1}{\sum}x_{i}\frac{t^{i}}{i!}\right)  ^{k}\exp\left(
\underset{i\geq1}{\sum}y_{i}\frac{t^{i}}{i!}\right)  ,\text{ }k=0,1,\ldots.
\label{9}%
\end{equation}

\begin{theorem}
\label{T3}We have%
\[
B_{n+r,k+r}^{\left(  r\right)  }\left(  a_{l};b_{l}\right)  =\underset
{\Lambda\left(  n,k,r\right)  }{\sum}\left[  \frac{n!}{k_{1}!k_{2}!\cdots
}\left(  \frac{a_{1}}{1!}\right)  ^{k_{1}}\left(  \frac{a_{2}}{2!}\right)
^{k_{2}}\cdots\right]  \left[  \frac{r!}{r_{0}!r_{1}!\cdots}\left(
\frac{b_{1}}{0!}\right)  ^{r_{0}}\left(  \frac{b_{2}}{1!}\right)  ^{r_{1}%
}\cdots\right]  ,
\]
where%
\[
\Lambda\left(  n,k,r\right)  =\left\{
\begin{array}
[c]{c}%
\left(  \mathbf{k,r}\right)  =\left(  \left(  k_{i}:i\geq1\right)  ;\left(
r_{i}:i\geq0\right)  \right)  :\\
\\
k_{i}\in\mathbb{N},\ r_{i}\in\mathbb{N},\ \underset{i\geq1}{\sum}%
k_{i}=k,\ \underset{i\geq0}{\sum}r_{i}=r,\text{\ }\underset{i\geq1}{\sum
}i\left(  k_{i}+r_{i}\right)  =n
\end{array}
\right\}  .
\]

\end{theorem}

\begin{proof}
Setting%
\begin{align*}
\pi\left(  n,k,j\right)   &  =\left\{  \mathbf{k=}\left(  k_{1},\ldots
,k_{n};r_{1},\ldots,r_{n}\right)  :\overset{n}{\underset{i=1}{\sum}}%
k_{i}=k,\ \overset{n}{\underset{i=1}{\sum}}r_{i}=j,\text{\ }\overset
{n}{\underset{i=1}{\sum}}i\left(  k_{i}+r_{i}\right)  =n\right\}  ,\\
\Pi\left(  n,k,r\right)   &  =\left\{  \mathbf{k=}\left(  k_{1},\ldots
,k_{n};r_{0},\ldots,r_{n}\right)  :\overset{n}{\underset{i=1}{\sum}}%
k_{i}=k,\ \overset{n}{\underset{i=0}{\sum}}r_{i}=r,\text{\ }\overset
{n}{\underset{i=1}{\sum}}i\left(  k_{i}+r_{i}\right)  =n\right\}  ,\\
T_{n,k,s}\left(  a_{l};b_{l+1}\right)   &  =\underset{\pi\left(  n,k,s\right)
}{\sum}\frac{n!}{k_{1}!\cdots k_{n}!r_{1}!\cdots r_{n}!}\left(  \frac{a_{1}%
}{1!}\right)  ^{k_{1}}\cdots\left(  \frac{a_{n}}{n!}\right)  ^{k_{n}}\left(
\frac{b_{2}}{1!}\right)  ^{r_{1}}\cdots\left(  \frac{b_{n+1}}{n!}\right)
^{r_{n}}.
\end{align*}
On using Corollary \ref{T1}, we obtain%
\[
\underset{n\geq k}{\sum}\left(  \exp\left(  -b_{1}u\right)  \underset{r\geq
0}{\sum}B_{n+r,k+r}^{\left(  r\right)  }\left(  a_{l};b_{l}\right)
\frac{u^{r}}{r!}\right)  \frac{t^{n}}{n!}=\frac{1}{k!}\left(  \underset
{j\geq1}{\sum}a_{j}\frac{t^{j}}{j!}\right)  ^{k}\exp\left(  u\underset{j\geq
1}{\sum}b_{j+1}\frac{t^{j}}{j!}\right)  .
\]
Upon using (\ref{9}), the last expression shows that%
\begin{align*}
&  \exp\left(  -b_{1}u\right)  \underset{r\geq0}{\sum}B_{n+r,k+r}^{\left(
r\right)  }\left(  a_{l};b_{l}\right)  \frac{u^{r}}{r!}\\
&  =T_{n,k}\left(  a_{1},\ldots,a_{n};ub_{2},\ldots,ub_{n+1}\right) \\
&  =\underset{\pi\left(  n,k\right)  }{\sum}\frac{n!}{k_{1}!\cdots k_{n}%
!r_{1}!\cdots r_{n}!}\left(  \frac{a_{1}}{1!}\right)  ^{k_{1}}\cdots\left(
\frac{a_{n}}{n!}\right)  ^{k_{n}}\left(  \frac{b_{2}}{1!}\right)  ^{r_{1}%
}\cdots\left(  \frac{b_{n+1}}{n!}\right)  ^{r_{n}}u^{r_{1}+\cdots+r_{n}}\\
&  =\underset{s\geq0}{\sum}u^{s}\underset{\pi\left(  n,k,s\right)  }{\sum
}\frac{n!s!}{k_{1}!\cdots k_{n}!r_{1}!\cdots r_{n}!}\left(  \frac{a_{1}}%
{1!}\right)  ^{k_{1}}\cdots\left(  \frac{a_{n}}{n!}\right)  ^{k_{n}}\left(
\frac{b_{2}}{1!}\right)  ^{r_{1}}\cdots\left(  \frac{b_{n+1}}{n!}\right)
^{r_{n}}\\
&  =\underset{s\geq0}{\sum}s!T_{n,k,s}\left(  a_{l};b_{l+1}\right)
\frac{u^{s}}{s!}.
\end{align*}
So, we obtain%
\begin{align*}
\underset{r\geq0}{\sum}B_{n+r,k+r}^{\left(  r\right)  }\left(  a_{l}%
;b_{l}\right)  \frac{u^{r}}{r!}  &  =\exp\left(  b_{1}u\right)  \underset
{r\geq0}{\sum}s!T_{n,k,s}\left(  a_{l};b_{l+1}\right)  \frac{u^{s}}{s!}\\
&  =\underset{r\geq0}{\sum}\frac{u^{r}}{r!}\underset{j=0}{\overset{r}{\sum}%
}\binom{r}{j}j!b_{1}^{r-j}T_{n,k,j}\left(  a_{l};b_{l+1}\right)  .
\end{align*}
Then%
\begin{align*}
&  B_{n+r,k+r}^{\left(  r\right)  }\left(  a_{l};b_{l}\right) \\
&  =\underset{j=0}{\overset{r}{\sum}}\binom{r}{j}b_{1}^{r-j}j!T_{n,k,j}\left(
a_{l};b_{l+1}\right) \\
&  =\underset{r_{0}=0}{\overset{r}{\sum}}\frac{b_{1}^{r_{0}}}{r_{0}!}%
\underset{\pi\left(  n,k,r_{0}-j\right)  }{\sum}\frac{n!r!}{k_{1}!\cdots
k_{n}!r_{1}!\cdots r_{n}!}\left(  \frac{a_{1}}{1!}\right)  ^{k_{1}}%
\cdots\left(  \frac{a_{n}}{n!}\right)  ^{k_{n}}\left(  \frac{b_{2}}%
{1!}\right)  ^{r_{1}}\cdots\left(  \frac{b_{n+1}}{n!}\right)  ^{r_{n}}\\
&  =\underset{\Pi\left(  n,k,r\right)  }{\sum}\frac{n!r!}{k_{1}!\cdots
k_{n}!r_{0}!r_{1}!\cdots r_{n}!}\left(  \frac{a_{1}}{1!}\right)  ^{k_{1}%
}\cdots\left(  \frac{a_{n}}{n!}\right)  ^{k_{n}}\left(  \frac{b_{1}}%
{0!}\right)  ^{r_{0}}\left(  \frac{b_{2}}{1!}\right)  ^{r_{1}}\cdots\left(
\frac{b_{n+1}}{n!}\right)  ^{r_{n}}.
\end{align*}
The elements of $\Lambda\left(  n,k,r\right)  $ can be reduced to those of
$\Pi\left(  n,k,r\right)  $ because we get necessarily $k_{j}=r_{j+1}=0$ for
$j\geq n+1.$ Thus, the expression of $B_{n+r,k+r}^{\left(  r\right)  }\left(
a_{l};b_{l}\right)  $ results.
\end{proof}

\section{Some properties of the partial $r$-Bell polynomials}

Other combinatorial processes give the following identity.

\begin{proposition}
We have%
\begin{equation}
B_{n+r,k+r}^{\left(  r\right)  }\left(  a_{1},a_{2},\ldots;b_{1},b_{2}%
,\ldots\right)  =\underset{i=0}{\overset{r}{\sum}}\underset{j=0}{\overset
{k}{\sum}}\binom{r}{i}\binom{n}{j}b_{1}^{i}a_{1}^{j}B_{n-j+r-i,k-j+r-i}%
^{\left(  r-i\right)  }\left(  0,a_{2},a_{3},\ldots;0,b_{2},b_{3}%
,\ldots\right)  . \label{1}%
\end{equation}

\end{proposition}

\begin{proof}
Consider the $\left(  n+r\right)  $-set as union of two sets $\mathbf{R}$
which contains the $r$ first elements and $\mathbf{N}$ which contains the $n$
last elements. Choice $i$ elements in $\mathbf{R}$ and $j$ elements in
$\mathbf{N}$ to form $i+j$ singletons. Because each singleton can be colored
with $b_{1}$ colors if it is in $\mathbf{R}$ and $a_{1}$ colors if it is in
$\mathbf{N},$ then, the number of the colored singletons is $\binom{r}%
{i}\binom{n}{j}b_{1}^{i}a_{1}^{j}.$ The elements not really used is of number
$r-i+n-j$ which can be partitioned into $r-i+k-j$ colored partitions with non
singletons (such that the $r-i$ first elements are in different blocks) in
$B_{n-j+r-i,k-j+r-i}^{\left(  r-i\right)  }\left(  0,a_{2},a_{3}%
,\ldots;0,b_{2},b_{3},\ldots\right)  $ ways. Then, for a fixed $i $ and a
fixed $j,$ there are $\binom{r}{i}\binom{n}{j}b_{1}^{i}a_{1}^{j}%
B_{n-j+r-i,k-j+r-i}^{\left(  r-i\right)  }\left(  0,a_{2},a_{3},\ldots
;0,b_{2},b_{3},\ldots\right)  $ colored partitions. So, the number of all
colored partitions is%
\[
\underset{i=0}{\overset{r}{\sum}}\underset{j=0}{\overset{k}{\sum}}\binom{r}%
{i}\binom{n}{j}b_{1}^{i}a_{1}^{j}B_{n-j+r-i,k-j+r-i}^{\left(  r-i\right)
}\left(  0,a_{2},a_{3},\ldots;0,b_{2},b_{3},\ldots\right)  =B_{n+r,k+r}%
^{\left(  r\right)  }\left(  a_{1},a_{2},\ldots;b_{1},b_{2},\ldots\right)  .
\]

\end{proof}

\noindent On using Corollary \ref{T1} or Theorem \ref{T3}, we can verity that

\begin{proposition}
We have%
\begin{align}
B_{n+r,k+r}^{\left(  r\right)  }\left(  xa_{l};yb_{l}\right)   &  =x^{k}%
y^{r}B_{n+r,k+r}^{\left(  r\right)  }\left(  a_{l};b_{l}\right)  ,\label{6}\\
B_{n+r,k+r}^{\left(  r\right)  }\left(  x^{l}a_{l};x^{l}b_{l}\right)   &
=x^{n+r}B_{n+r,k+r}^{\left(  r\right)  }\left(  a_{l};b_{l}\right)
,\label{7}\\
B_{n+r,k+r}^{\left(  r\right)  }\left(  x^{l-1}a_{l};x^{l-1}b_{l}\right)   &
=x^{n-k}B_{n+r,k+r}^{\left(  r\right)  }\left(  a_{l};b_{l}\right)  .
\label{8}%
\end{align}

\end{proposition}

\noindent The relations of the following proposition generalize some of the
known relations on partial Bell polynomials.

\begin{proposition}
\label{P0}We have%
\begin{align*}
\overset{n}{\underset{j=1}{\sum}}\binom{n}{j}a_{j}B_{n+r-j,k+r-1}^{\left(
r\right)  }\left(  a_{l};b_{l}\right)   &  =kB_{n+r,k+r}^{\left(  r\right)
}\left(  a_{l};b_{l}\right)  ,\\
\overset{n}{\underset{j=1}{\sum}}\binom{n}{j-1}b_{j}B_{n-j+r-1,k+r-1}^{\left(
r-1\right)  }\left(  a_{l};b_{l}\right)   &  =rB_{n+r,k+r}^{\left(  r\right)
}\left(  a_{l};b_{l}\right)
\end{align*}
and%
\[
\overset{n}{\underset{j=1}{\sum}}ja_{j}\binom{n}{j}B_{n+r-j,k+r-1}^{\left(
r\right)  }\left(  a_{l};b_{l}\right)  +r\overset{n}{\underset{j=1}{\sum}%
}jb_{j}\binom{n}{j-1}B_{n-j+r-1,k+r-1}^{\left(  r-1\right)  }\left(
a_{l};b_{l}\right)  =\left(  n+r\right)  B_{n+r,k+r}^{\left(  r\right)
}\left(  a_{l};b_{l}\right)  .
\]

\end{proposition}

\begin{proof}
On using Corollary \ref{T1}, we deduce that%
\begin{align*}
\frac{\partial}{\partial a_{j}}B_{n+r,k+r}^{\left(  r\right)  }\left(
a_{l};b_{l}\right)   &  =\binom{n}{j}B_{n-j+r,k-1+r}^{\left(  r\right)
}\left(  a_{l};b_{l}\right)  ,\ \ \\
\frac{\partial}{\partial b_{j}}B_{n+r,k+r}^{\left(  r\right)  }\left(
a_{l};b_{l}\right)   &  =\binom{n}{j-1}B_{n-j+r-1,k+r-1}^{\left(  r-1\right)
}\left(  a_{l};b_{l}\right)  .
\end{align*}
Then, by derivation the two sides of (\ref{6}) in first time respect to $x$
and in second time respect to $y,$ we obtain%
\begin{align*}
\overset{n}{\underset{j=1}{\sum}}\binom{n}{j}a_{j}B_{n+r-j,k+r-1}^{\left(
r\right)  }\left(  a_{l}x;yb_{l}\right)   &  =kx^{k-1}y^{r}B_{n+r,k+r}%
^{\left(  r\right)  }\left(  a_{l};b_{l}\right)  ,\\
\overset{n}{\underset{j=1}{\sum}}\binom{n}{j-1}b_{j}B_{n-j+r-1,k+r-1}^{\left(
r-1\right)  }\left(  a_{l}x;yb_{l}\right)   &  =rx^{k}y^{r-1}B_{n+r,k+r}%
^{\left(  r\right)  }\left(  a_{l};b_{l}\right)  ,
\end{align*}
and by derivation the two sides of (\ref{7}) respect to $x,$ we obtain%
\begin{align*}
&  \overset{n}{\underset{j=1}{\sum}}jx^{j-1}a_{j}\binom{n}{j}B_{n+r-j,k+r-1}%
^{\left(  r\right)  }\left(  a_{l}x^{l};b_{l}y^{l}\right)  +r\overset
{n}{\underset{j=1}{\sum}}jx^{j-1}b_{j}\binom{n}{j-1}B_{n-j+r-1,k+r-1}^{\left(
r-1\right)  }\left(  a_{l}x^{l};b_{l}y^{l}\right) \\
&  =\left(  n+r\right)  x^{n+r-1}B_{n+r,k+r}^{\left(  r\right)  }\left(
a_{l};b_{l}\right)  .
\end{align*}
The three relations of the proposition follow by taking $x=y=1.$
\end{proof}

\noindent The partial $r$-Bell polynomials can be expressed by the partial
bell polynomials as follows.

\begin{proposition}
\label{P2}We have%
\[
B_{n+r,k+r}^{\left(  r\right)  }\left(  a_{l};b_{l}\right)  =\binom{n+r}%
{r}^{-1}\underset{j=k}{\overset{n}{\sum}}\binom{n+r}{j}B_{j,k}\left(
a_{l}\right)  B_{n+r-j,r}\left(  lb_{l}\right)  .
\]

\end{proposition}

\begin{proof}
This proposition follows from the expansion%
\begin{align*}
t^{r}\underset{n\geq k}{\sum}B_{n+r,k+r}^{\left(  r\right)  }\left(
a_{l};b_{l}\right)  \frac{t^{n}}{n!}  &  =\frac{t^{r}}{k!}\left(
\underset{j\geq1}{\sum}a_{j}\frac{t^{j}}{j!}\right)  ^{k}\left(
\underset{j\geq0}{\sum}b_{j+1}\frac{t^{j}}{j!}\right)  ^{r}\\
&  =\frac{1}{k!}\left(  \underset{j\geq1}{\sum}a_{j}\frac{t^{j}}{j!}\right)
^{k}\left(  \underset{j\geq1}{\sum}jb_{j}\frac{t^{j}}{j!}\right)  ^{r}\\
&  =r!\left(  \underset{i\geq k}{\sum}B_{i,k}\left(  a_{l}\right)  \frac
{t^{i}}{i!}\right)  \left(  \underset{j\geq r}{\sum}B_{j,r}\left(
lb_{l}\right)  \frac{t^{j}}{j!}\right)  .
\end{align*}

\end{proof}

\begin{proposition}
We have%
\[
\binom{n}{r}B_{n+k-r,k+r}^{\left(  r\right)  }\left(  la_{l};b_{l}\right)
=\binom{n}{k}B_{n-k+r,k+r}^{\left(  k\right)  }\left(  lb_{l};a_{l}\right)
,\ \ n\geq2\max\left(  k,r\right)  .
\]

\end{proposition}

\begin{proof}
From Corollary \ref{T1} we get%
\[
\frac{t^{r}}{r!}\underset{n\geq k}{\sum}B_{n+k,k+r}^{\left(  r\right)
}\left(  la_{l};b_{l}\right)  \frac{t^{n}}{n!}=\frac{1}{k!r!}\left(
\underset{j\geq1}{\sum}ja_{j}\frac{t^{j}}{j!}\right)  ^{k}\left(
\underset{j\geq1}{\sum}jb_{j}\frac{t^{j}}{j!}\right)  ^{r}%
\]
and by the symmetry respect to $\left(  k,\left(  a_{j}\right)  \right)  $ and
$\left(  r,\left(  b_{j}\right)  \right)  $ in the last expression we get%
\[
\frac{t^{r}}{r!}\underset{n\geq k}{\sum}B_{n+k,k+r}^{\left(  r\right)
}\left(  la_{l};b_{l}\right)  \frac{t^{n}}{n!}=\frac{t^{k}}{k!}\underset{n\geq
r}{\sum}B_{n+r,k+r}^{\left(  k\right)  }\left(  lb_{l};a_{l}\right)
\frac{t^{n}}{n!}.
\]
So, we obtain the desired identity.
\end{proof}

\section{New combinatorial interpretations of the $r$-Whitney numbers}

The $r$-Whitney numbers of both kinds $w_{m,r}\left(  n,k\right)  $ and
$W_{m,r}\left(  n,k\right)  $ are introduced by Mez\H{o} \cite{mez} and the
$r$-Whitney-Lah numbers $L_{m,r}\left(  n,k\right)  $ are introduced by Cheon
and Jung \cite{che, rah}. Some of the properties of these numbers are given in
\cite{che}. In this paragraph we use the combinatorial interpretation of the
partial $r$-Bell polynomials given above to deduce a new combinatorial
interpretations for the numbers $\left\vert w_{m,r}\left(  n,k\right)
\right\vert ,$ $W_{m,r}\left(  n,k\right)  $ and $L_{m,r}\left(  n,k\right)
.$

\noindent The $r$-Whitney numbers of the first kind $w_{m,r}\left(
n,k\right)  $ are given by their generating function%
\[
\underset{n\geq k}{\sum}w_{m,r}\left(  n,k\right)  \frac{t^{n}}{n!}=\frac
{1}{k!}\left(  \ln\left(  1+mt\right)  \right)  ^{k}\left(  \left(
1+mt\right)  ^{-\frac{1}{m}}\right)  ^{r}.
\]
So that
\[
w_{m,r}\left(  n,k\right)  =\left(  -1\right)  ^{n-k+r}B_{n+r,k+r}^{\left(
r\right)  }\left(  \left(  l-1\right)  !m^{l-1};\left(  m+1\right)  \left(
2m+1\right)  \cdots\left(  \left(  l-1\right)  m+1\right)  \right)  .
\]
This means that the absolute $r$-Whitney number of the first kind $\left\vert
w_{m,r}\left(  n,k\right)  \right\vert $ counts the number of partitions of a
$n$-set into $k$ blocks such that\newline- the $r$ first elements are in
different blocks, \newline- any block of cardinality $i$ and no contain an
element of the $r$ first elements can be colored with $\left(  i-1\right)
!m^{i-1}$ colors, and, \newline- any block of cardinality $i$ and contain one
element of the $r$ first elements can be colored with $\left(  m+1\right)
\left(  2m+1\right)  \cdots\left(  \left(  i-1\right)  m+1\right)  $ colors.

\noindent The $r$-Whitney numbers of the second kind $W_{m,r}\left(
n,k\right)  $ are given by their generating function%
\[
\underset{n\geq k}{\sum}W_{m,r}\left(  n,k\right)  \frac{t^{n}}{n!}=\frac
{1}{k!}\left(  \frac{\exp\left(  mt\right)  -1}{m}\right)  ^{k}\exp\left(
rt\right)  .
\]
So that
\[
W_{m,r}\left(  n,k\right)  =B_{n+r,k+r}^{\left(  r\right)  }\left(
m^{l-1};1\right)  .
\]
This means that the $r$-Whitney number of the second kind $W_{m,r}\left(
n,k\right)  $ counts the number of partitions of a $n$-set into $k$ blocks
such that \newline- the $r$ first elements are in different blocks, \newline-
any block of cardinality $i$ and no contain any element of the $r$ first
elements can be colored with $m^{i-1}$ colors, and, \newline- any block of
cardinality $i$ and contain one element of the $r$ first elements can be
colored with one color.

\noindent The $r$-Whitney-Lah numbers $L_{m,r}\left(  n,k\right)  $ are given
by their generating function%
\[
\underset{n\geq k}{\sum}L_{m,r}\left(  n,k\right)  \frac{t^{n}}{n!}=\frac
{1}{k!}\left(  t\left(  1-mt\right)  ^{-1}\right)  ^{k}\left(  \left(
1-mt\right)  ^{-\frac{2}{m}}\right)  ^{r}.
\]
So that
\[
L_{m,r}\left(  n,k\right)  =B_{n+r,k+r}^{\left(  r\right)  }\left(
l!m^{l-1};2\left(  m+2\right)  \cdots\left(  \left(  l-1\right)  m+2\right)
\right)  .
\]
This means that the $r$-Whitney number of the second kind $L_{m,r}\left(
n,k\right)  $ counts the number of partitions of a $n$-set into $k$ blocks
such that \newline- the $r$ first elements are in different blocks, \newline-
any block no contain an element of the $r$ first elements and is of length $i$
can be colored with $i!m^{i-1}$ colors, and, \newline- any block of
cardinality $i$ and contain one element of the $r$ first elements can be
colored with $2\left(  m+2\right)  \cdots\left(  \left(  i-1\right)
m+2\right)  $ colors.

\section{Application to the $r$-Stirling numbers of the second kind}

From Definition \ref{D1} we may state that the number $B_{n,k}^{\left(
r\right)  }\left(  a_{l}\right)  :=B_{n,k}^{\left(  r\right)  }\left(
a_{l},a_{l}\right)  $ counts the number of partitions of $n$-set into $k$
blocks such that the blocks of the same cardinality $i$ can be colored with
$a_{i}$ colors (such cycles with $a_{i}=0$ does not exist) and the $r$ first
elements are in different blocks.

\noindent For $a_{n}=1,$ $n\geq1,$ we get the know $r$-Stirling numbers of the
second kind%
\[%
\genfrac{\{}{\}}{0pt}{}{n}{k}%
_{r}=B_{n,k}^{\left(  r\right)  }\left(  1,1,\cdots\right)
\]
which counts the number of partitions of a $n$-set into $k$ blocks such that
the $r$ first elements are in different blocks.

\noindent For $a_{n}=1,$ $n\geq m$ and $a_{n}=0,$ $n\leq m-1,$ we get the
$m$-associated $r$-Stirling numbers of the second kind%
\[%
\genfrac{\{}{\}}{0pt}{}{n}{k}%
_{r}^{m\uparrow}=B_{n,k}^{\left(  r\right)  }\left(  \overset{m-1}%
{\overbrace{0,\cdots,0}},1,1,\cdots\right)
\]
which counts the number of partitions of a $n$-set into $k$ blocks such that
the cardinality of any block is at least $m$ elements and the $r$ first
elements are in different blocks.

\noindent For $a_{n}=0,$ $n\geq m+1$ and $a_{n}=1,$ $n\leq m,$ we get the
$m$-truncated $r$-Stirling numbers of the second kind%
\[%
\genfrac{\{}{\}}{0pt}{}{n}{k}%
_{r}^{m\downarrow}=B_{n,k}^{\left(  r\right)  }\left(  \overset{m}%
{\overbrace{1,\cdots,1}},0,0,\cdots\right)
\]
which counts the number of partitions of a $n$-set into $k$ blocks such that
the cardinality of any block is $\leq m$ elements and the $r$ first elements
are in different blocks.

\noindent For $a_{2n-1}=0$ and $a_{2n}=1,$ $n\geq1,$ we get the $r$-Stirling
numbers of the second kind in even parts%
\[%
\genfrac{\{}{\}}{0pt}{}{n}{k}%
_{r}^{even}=B_{n,k}^{\left(  r\right)  }\left(  0,1,0,1,0,\cdots\right)
\]
which counts the number of partitions of a $n$-set into $k$ blocks such that
the cardinality of any block is even and the $r$ first elements are in
different blocks.

\noindent For $a_{2n-1}=1$ and $a_{2n}=0,$ $n\geq1,$ we get the $r$-Stirling
numbers of the second kind in odd parts%
\[%
\genfrac{\{}{\}}{0pt}{}{n}{k}%
_{r}^{odd}=B_{n,k}^{\left(  r\right)  }\left(  1,0,1,0,\cdots\right)
\]
which counts the number of partitions of a $n$-set into $k$ blocks such that
the cardinality of any block is odd and the $r$ first elements are in
different blocks.

\section{Application to the $r$-Stirling numbers of the first kind}

We start this application by giving a second combinatorial interpretation of
the partial $r$-Bell polynomials.

\begin{proposition}
The number $B_{n,k}^{\left(  r\right)  }\left(  \left(  l-1\right)
!a_{l}\right)  :=B_{n,k}^{\left(  r\right)  }\left(  \left(  l-1\right)
!a_{l},\left(  l-1\right)  !a_{l}\right)  $ counts the number of permutations
of a $n$-set into $k$ cycles such that the cycles of the same length $i$ can
be colored with $a_{i}$ colors (such cycles with $a_{i}=0$ does not exist) and
the $r$ first elements are in different cycles.
\end{proposition}

\begin{proof}
Let $\Pi_{r,k}$ be the set of partitions $\pi$ of the set $n$-set into $k$
blocks such that the blocks of the same cardinality $i$ posses $\left(
i-1\right)  !a_{i}$ colors and the $r$ first elements are in different blocks,
and, $P_{r,k}$ be the set of permutations $P$ of the elements of the set
$n$-set into $k$ cycles such that the cycles of the same length $i$ can be
colored with $a_{i}$ colors the $r$ first elements are in different cycles.
The application $\varphi:\Pi_{r,k}\rightarrow P_{r,k}$ which associate any
(colored) partition $\pi$ of $\Pi_{r,k},$ $\pi=S_{1}\cup\cdots\cup S_{k},$
$1\leq k\leq n,$ a (colored) permutation $P$ of $P_{r,k},$ $P=C_{1}\cup
\cdots\cup C_{k},$ such that the elements of $C_{i}$ are exactly those of
$S_{i}.$ It is obvious that the application $\varphi$ is bijective.
\end{proof}

\noindent For $a_{n}=\left(  n-1\right)  !,$ $n\geq1,$ we get the$\ $known
$r$-Stirling numbers of the first kind%
\[%
\genfrac{[}{]}{0pt}{}{n}{k}%
_{r}=B_{n,k}^{\left(  r\right)  }\left(  0!,1!,2!,\cdots\right)
\]
which counts the number of permutations of a $n$-set into $k$ cycles such that
the $r$ first elements are in different cycles.

\noindent For $a_{n}=\left(  n-1\right)  !,$ $n\geq m$ and $a_{n}=0,$ $n\leq
m-1,$ we get the $m$-associated $r$-Stirling numbers of the first kind%
\[%
\genfrac{[}{]}{0pt}{}{n}{k}%
_{r}^{m\uparrow}=B_{n,k}^{\left(  r\right)  }\left(  \overset{m-1}%
{\overbrace{0,\cdots,0}},\left(  m-1\right)  !,m!,\cdots\right)
\]
which counts the number of permutations of a $n$-set into $k$ cycles such that
the length of any cycle is equal at least $m$ and the $r$ first elements are
in different cycles.

\noindent For $a_{n}=\left(  n-1\right)  !,$ $n\leq m$ and $a_{n}=0,$ $n\geq
m+1,$ we get the $m$-truncated $r$-Stirling numbers of the first kind%
\[%
\genfrac{[}{]}{0pt}{}{n}{k}%
_{r}^{m\downarrow}=B_{n,k}^{\left(  r\right)  }\left(  \overset{m}%
{\overbrace{0!,\cdots,\left(  m-1\right)  !}},0,0,\cdots\right)
\]
which counts the number of permutations of a $n$-set into $k$ cycles such that
the length of any cycle is $\leq m$ and the $r$ first elements are in
different cycles.

\noindent For $a_{2n-1}=0$ and $a_{2n}=\left(  2n-1\right)  !,$ $n\geq1,$ we
get the $r$-Stirling numbers of the first kind in cycles of even lengths%
\[%
\genfrac{[}{]}{0pt}{}{n}{k}%
_{r}^{even}=B_{n,k}^{\left(  r\right)  }\left(  0,1!,0,3!,0,5!,0,\cdots
\right)
\]
which counts the number of permutations of a $n$-set into $k$ cycles such that
the length of any cycle is even and the $r$ first elements are in different cycles.

\noindent For $a_{2n-1}=\left(  2n-2\right)  !$ and $a_{2n}=0,$ $n\geq1,$ we
get the $r$-Stirling numbers of the first kind in cycles of odd lengths%
\[%
\genfrac{[}{]}{0pt}{}{n}{k}%
_{r}^{odd}=B_{n,k}^{\left(  r\right)  }\left(  0!,0,2!,0,4!,0,\cdots\right)
\]
which counts the number of permutations of a $n$-set into $k$ cycles such that
the length of any cycle is odd and the $r$ first elements are in different cycles.

\section{Application to the $r$-Lah numbers}

\noindent Then, similarly to the partial Bell polynomials, we establish the
following (third) combinatorial interpretation of the partial $r$-Bell polynomials.

\begin{proposition}
The number $B_{n,k}^{\left(  r\right)  }\left(  l!a_{l}\right)  :=B_{n,k}%
^{\left(  r\right)  }\left(  l!a_{l},l!a_{l}\right)  $ counts the number of
partitions of a $n$-set into $k$ ordered blocks such that the blocks of the
same cardinality $i$ can be colored with $a_{i}$ colors (such block with
$a_{i}=0$ does not exist) and the $r$ first elements are in different blocks.
\end{proposition}

\begin{proof}
Let $\Pi_{r,k}^{\prime}$ be the set of partitions $\pi$ of the a $n$-set into
$k$ blocks such that the blocks of the same cardinality $i$ can be colored
with $i!a_{i}$ colors and the $r$ first elements are in different blocks, and,
$\Pi_{r,k}^{ord}$ be the set of partitions $\pi^{ord}$ of the a $n$-set into
$k$ ordered blocks such that the blocks of the same length $i$ can be colored
with $a_{i}$ the $r$ first elements are in different blocks. The application
$\varphi:\Pi_{r,k}^{\prime}\rightarrow\Pi_{r,k}^{ord}$ which associate a
(colored) partition $\pi$ of $\Pi_{r,k}^{\prime},$ $\pi=S_{1}\cup\cdots\cup
S_{k},$ $1\leq k\leq n,$ a (colored) partition of ordered blocks $\pi^{ord}$
of $\Pi_{r}^{ord},$ $\pi^{ord}=P_{1}\cup\cdots\cup P_{k},$ such that the
elements of $P_{i}$ are exactly those of $S_{i}.$ It is obvious that the
application $\varphi$ is bijective.
\end{proof}

\noindent For $a_{n}=n!,$ $n\geq1$, we get the $r$-Lah numbers%
\[%
\genfrac{\lfloor}{\rfloor}{0pt}{}{n}{k}%
_{r}=B_{n,k}^{\left(  r\right)  }\left(  1!,2!,3!,\cdots\right)  .
\]
The $\ r$-Lah number $%
\genfrac{\lfloor}{\rfloor}{0pt}{}{n}{k}%
_{r}$ counts the number of partitions of a $n$-set into $k$ ordered blocks
such that the $r$ first elements are in different blocks.

\noindent For $a_{n}=0,$ $n\leq m-1,$ and $a_{n}=n!,$ $n\geq m,$ we get the
$m$-degenerate $r$-Lah numbers%
\[%
\genfrac{\lfloor}{\rfloor}{0pt}{}{n}{k}%
_{r}^{m\uparrow}=B_{n,k}^{\left(  r\right)  }\left(  \overset{m-1}%
{\overbrace{0,\cdots,0}},m!,\left(  m+1\right)  !,\cdots\right)
\]
which counts the number of partitions of a $n$-set into $k$ ordered blocks
such that the cardinality of any block is $\geq m$ and the $r$ first elements
are in different blocks.

\noindent For $a_{n}=n!,$ $n\leq m,$ and $a_{n}=0,$ $n\geq m+1,$ we get the
$m$-truncated $r$-Lah numbers%
\[%
\genfrac{\lfloor}{\rfloor}{0pt}{}{n}{k}%
_{r}^{m\downarrow}=B_{n,k}^{\left(  r\right)  }\left(  \overset{m}%
{\overbrace{1!,\cdots,m!}},0,0,\cdots\right)
\]
which counts the number of partitions of a $n$-set into $k$ ordered blocks
such that the cardinality of any block is $\leq m$ and the $r$ first elements
are in different blocks.

\noindent For $a_{2n-1}=0$ and $a_{2n}=\left(  2n\right)  !,$ $n\geq1,$ we get
the $r$-Lah numbers in blocks of even cardinalities%
\[%
\genfrac{\lfloor}{\rfloor}{0pt}{}{n}{k}%
_{r}^{even}=B_{n,k}^{\left(  r\right)  }\left(  0,2!,0,4!,0,6!,0,\cdots
\right)  ,
\]
which represents the number of partitions of a $n$-set into $k$ ordered blocks
such that the cardinality of any block is even and the $r$ first elements are
in different blocks.

\noindent For $a_{2n-1}=\left(  2n-1\right)  !$ and $a_{2n}=0,$ $n\geq1,$ we
get the $r$-Lah numbers in blocks of even cardinalities%
\[%
\genfrac{\lfloor}{\rfloor}{0pt}{}{n}{k}%
_{r}^{odd}=B_{n,k}^{\left(  r\right)  }\left(  1!,0,3!,0,5!,0,\cdots\right)
\]
which represents the number of partitions of a $n$-set into $k$ ordered blocks
such that the cardinality of any block is odd and the $r$ first elements are
in different blocks.

\section{Application to sum of independent random variables}

It is known that for a sequence of independent random variables $\left\{
X_{n}\right\}  $ with all its moments exist and are the same, $\mu
_{n}=\operatorname*{E}\left(  X^{n}\right)  $ we have $\operatorname*{E}%
\left(  S_{p}^{n}\right)  =\binom{n+p}{p}^{-1}B_{n+p,p}\left(  l\mu
_{l-1}\right)  .$ The following theorem generalize this result.

\begin{theorem}
Let $\left\{  X_{n}\right\}  $ and $\left\{  Y_{n}\right\}  $ be two
independent sequences of independent random variables with all their moments
exist and are the same, $\mu_{n}=\operatorname*{E}\left(  X^{n}\right)
,\ \nu_{n}=\operatorname*{E}\left(  Y^{n}\right)  $ and let%
\[
S_{p,q}=X_{1}+\cdots+X_{p}+Y_{1}+\cdots+Y_{q}.
\]
Then we have%
\[
\operatorname*{E}\left(  S_{p,q}^{n}\right)  =\binom{n+p}{p}^{-1}%
B_{n+p+q,p+q}^{\left(  q\right)  }\left(  l\mu_{l-1},\nu_{l-1}\right)  .
\]

\end{theorem}

\begin{proof}
Let $\varphi_{X}\left(  t\right)  $ be the common generating function of
moments for $X_{n},\ n\geq1,$ $\varphi_{Y}\left(  t\right)  $ be the common
generating function of moments for $Y_{n},\ n\geq1,$ and, $\varphi_{S_{p,q}%
}\left(  t\right)  $ be the generating function of moments of $S_{p,q} $.
Then, in first part, we get%
\begin{align*}
t^{p}\varphi_{S_{p,q}}\left(  t\right)   &  =\operatorname*{E}\left(
t^{p}\exp\left(  tS_{p,q}\right)  \right) \\
&  =\left(  \operatorname*{E}\left(  t\exp\left(  tX_{1}\right)  \right)
\right)  ^{p}\left(  \operatorname*{E}\left(  \exp\left(  tY_{1}\right)
\right)  \right)  ^{q}\\
&  =\left(  \underset{j\geq1}{\sum}j\mu_{j-1}\frac{t^{j}}{j!}\right)
^{p}\left(  \underset{j\geq0}{\sum}\nu_{j}\frac{t^{j}}{j!}\right)  ^{q}\\
&  =p!\underset{n\geq p}{\sum}B_{n+q,p+q}^{\left(  q\right)  }\left(
l\mu_{l-1},\nu_{l-1}\right)  \frac{t^{n}}{n!},
\end{align*}
and, in second part, we have%
\[
t^{p}\varphi_{S_{p,q}}\left(  t\right)  =\underset{j\geq0}{\sum}%
\operatorname*{E}\left(  S_{p,q}^{j}\right)  \frac{t^{j+p}}{j!}=\underset
{n\geq p}{\sum}\frac{n!}{\left(  n-p\right)  !}\operatorname*{E}\left(
S_{p,q}^{n-p}\right)  \frac{t^{n}}{n!}.
\]

\end{proof}

\noindent For the choice $q_{0}=1$ and $q_{j}=0$ if $j\geq1$ in the last
theorem, we may state that:

\begin{corollary}
Let $\left\{  X_{n}\right\}  $ be a sequence of independent random variables
with all their moments exist and are the same, $\mu_{n}=\operatorname*{E}%
\left(  X^{n}\right)  $ and%
\[
S_{p,q}=X_{1}+\cdots+X_{p}+q.
\]
Then we have%
\[
\operatorname*{E}\left(  S_{p,q}^{n}\right)  =\binom{n+p}{p}^{-1}%
B_{n+p+q,p+q}^{\left(  q\right)  }\left(  l\mu_{l-1},1\right)  .
\]

\end{corollary}

\begin{example}
Let $\left\{  X_{n}\right\}  $ be a sequence of independent random variables
with the same law of probability $\mathcal{U}\left(  0,1\right)  .$%
\[
S_{p,q}=X_{1}+\cdots+X_{p}+r.
\]
Then we have%
\[%
\genfrac{\{}{\}}{0pt}{}{n+p+r}{p+r}%
_{r}=\binom{n+p}{p}\operatorname*{E}\left(  S_{p,r}^{n}\right)  .
\]

\end{example}

\noindent It is also known that for a sequence of independent discrete random
variables $\left\{  X_{n}\right\}  $ with the same law of probability
$p_{j}:=P\left(  X_{1}=j\right)  ,\ \ j\geq0$ we have $P\left(  S_{p}%
=n\right)  =\frac{p!}{\left(  n+p\right)  !}B_{n+p,p}\left(  l!p_{l-1}\right)
.$ The following theorem generalize this result.

\begin{theorem}
Let $\left\{  X_{n}\right\}  $ and $\left\{  Y_{n}\right\}  $ be two
independent sequences of independent random variables with $p_{j}:=P\left(
X_{n}=j\right)  ,$ $q_{j}:=P\left(  Y_{n}=j\right)  $ and let%
\[
S_{p,q}=X_{1}+\cdots+X_{p}+Y_{1}+\cdots+Y_{q}.
\]
Then we have%
\[
P\left(  S_{p,q}=n\right)  =\frac{p!}{\left(  n+p\right)  !}B_{n+p+q,p+q}%
^{\left(  q\right)  }\left(  l!p_{l-1},\left(  l-1\right)  !q_{l-1}\right)  .
\]

\end{theorem}

\begin{proof}
It suffices to take in the last theorem $q_{0}=1$ and $q_{j}=0$ if $j\geq1.$%
\begin{align*}
\underset{n\geq p}{\sum}P\left(  S_{p,q}=n-p\right)  t^{n}  &  =t^{p}%
\underset{s\geq0}{\sum}P\left(  S_{p,q}=s\right)  t^{s}\\
&  =t^{p}\operatorname*{E}\left(  t^{S_{p,q}}\right) \\
&  =\left(  \underset{j\geq1}{\sum}p_{j-1}t^{j}\right)  ^{p}\left(
\underset{j\geq0}{\sum}q_{j}t^{j}\right)  ^{q}\\
&  =p!\underset{n\geq p}{\sum}B_{n+q,p+q}^{\left(  q\right)  }\left(
l!p_{l-1},\left(  l-1\right)  !q_{l-1}\right)  \frac{t^{n}}{n!}.
\end{align*}
This gives $P\left(  S_{p,q}=n\right)  =\frac{p!}{\left(  n+p\right)
!}B_{n+p+q,p+q}^{\left(  q\right)  }\left(  l!p_{l-1},\left(  l-1\right)
!q_{j-1}\right)  .$
\end{proof}

\noindent For the choice $q_{0}=1$ and $q_{j}=0$ if $j\geq1$ in the last
theorem, we may state that:

\begin{corollary}
Let $\left\{  X_{n}\right\}  $ be a sequence of independent discrete random
variables with the same law of probability $p_{j}:=P\left(  X_{1}=j\right)  $
and%
\[
S_{p,q}=X_{1}+\cdots+X_{p}+q.
\]
Then we have
\[
P\left(  S_{p,q}=n\right)  =\frac{p!}{\left(  n+p\right)  !}B_{n+p+q,p+q}%
\left(  l!p_{l-1},l!p_{l-1}\right)  .
\]

\end{corollary}

\section{Application on the successive derivatives of a function}

Let $F\left(  x\right)  =\underset{n\geq0}{\sum}f_{n}\frac{x^{n}}{n!}\in
C^{\infty}\left(  0\right)  $ and $G\left(  x\right)  =\underset{n\geq1}{\sum
}g_{n}\frac{\left(  x-a\right)  ^{n}}{n!}\in C^{\infty}\left(  a\right)  .$ It
is shown in \cite{com} that%
\[
\left.  \frac{d^{n}}{dx^{n}}\left(  F\left(  G\left(  x\right)  \right)
\right)  \right\vert _{x=a}=\underset{k=0}{\overset{n}{\sum}}f_{k}%
B_{n,k}\left(  g_{j}\right)  .
\]
The following theorem gives a similar result on using the partial $r$-Bell polynomials.

\begin{theorem}
Let $F,$ $G$ be as above and $H\left(  x\right)  =\underset{n\geq1}{\sum}%
h_{n}\frac{\left(  x-a\right)  ^{n}}{n!}\in C^{\infty}\left(  a\right)  .$
Then, we have%
\[
\left.  \frac{d^{n}}{dx^{n}}\left(  \left(  \frac{d}{dx}H\left(  x\right)
\right)  ^{r}F\left(  G\left(  x\right)  \right)  \right)  \right\vert
_{x=a}=\underset{k=0}{\overset{n}{\sum}}f_{k}B_{n+r,k+r}^{\left(  r\right)
}\left(  g_{j},h_{j}\right)  .
\]

\end{theorem}

\begin{proof}
This follows from%
\begin{align*}
&  \underset{n\geq0}{\sum}\left(  \underset{k=0}{\overset{n}{\sum}}%
f_{k}B_{n+r,k+r}^{\left(  r\right)  }\left(  g_{j},h_{j}\right)  \right)
\frac{\left(  x-a\right)  ^{n}}{n!}\\
&  =\underset{k\geq0}{\sum}f_{k}\underset{n\geq k}{\sum}B_{n+r,k+r}^{\left(
r\right)  }\left(  g_{j},h_{j}\right)  \frac{\left(  x-a\right)  ^{n}}{n!}\\
&  =\left(  \underset{j\geq0}{\sum}h_{j+1}\frac{\left(  x-a\right)  ^{j}}%
{j!}\right)  ^{r}\underset{k\geq0}{\sum}\frac{f_{k}}{k!}\left(  \underset
{j\geq1}{\sum}g_{j}\frac{\left(  x-a\right)  ^{j}}{j!}\right)  ^{k}\\
&  =\left(  \frac{d}{dx}H\left(  x\right)  \right)  ^{r}F\left(  G\left(
x\right)  \right)  .
\end{align*}

\end{proof}

\noindent For the choice $F\left(  x\right)  =\exp\left(  x\right)  ,$ we obtain:

\begin{corollary}
\label{C0}For $G,H\in C^{\infty}\left(  a\right)  $ with $G\left(  0\right)
=0, $ we have%
\[
\frac{d^{n}}{da^{n}}\left(  \left(  \frac{d}{da}H\left(  a\right)  \right)
^{r}\exp\left(  G\left(  a\right)  \right)  \right)  =\exp\left(  G\left(
a\right)  \right)  \underset{k=0}{\overset{n}{\sum}}B_{n+r,k+r}^{\left(
r\right)  }\left(  \frac{d^{j}}{da^{j}}G\left(  a\right)  ,\frac{d^{j}}%
{da^{j}}H\left(  a\right)  \right)  .
\]

\end{corollary}

\begin{example}
Let $G\left(  a\right)  =\frac{\exp\left(  ma\right)  -1}{m}$ and $H\left(
a\right)  =\exp\left(  a\right)  .$ On using Corollary \ref{C0} and the
generating function of the numbers $W_{m,r}\left(  n,k\right)  $ given above,
we get%
\[
\frac{d^{n}}{da^{n}}\left(  \exp\left(  \frac{\exp\left(  ma\right)  }%
{m}+ra\right)  \right)  =\exp\left(  \frac{\exp\left(  ma\right)  }%
{m}+ra\right)  \underset{k=0}{\overset{n}{\sum}}W_{m,r}\left(  n,k\right)
\exp\left(  mak\right)  .
\]

\end{example}

\end{document}